\documentclass[a4paper,10pt]{amsart}
\usepackage{amsmath,amssymb,amsthm,amscd}
\usepackage[all]{xy}
\renewcommand{\iff}{if and only if }
\newcommand{\st}{such that }

\newcommand{\ModR}{\hbox{{\rm Mod-}R}}

\newcommand{\Add}{\mathrm{Add}}
\newcommand{\Prod}{\mathrm{Prod}}

\DeclareMathOperator{\Hom}{Hom}

\DeclareMathOperator{\Ker}{Ker}

\usepackage[usenames]{color}

\theoremstyle{plain}
\newtheorem{thm}{Theorem}[section]
\newtheorem{prop}[thm]{Proposition}

\newtheorem{conj}[thm]{Conjecture}
\newtheorem{cor}[thm]{Corollary}
\theoremstyle{definition}
\newtheorem{defn}[thm]{Definition}
\newtheorem{exm}[thm]{Example}

\theoremstyle{remark}
\newtheorem*{rem}{Remark}

\begin{document}

%\subjclass[msc2010]{03C60 (primary), 03E55, 03C20, 16D10, 16D40 (secondary)}

\title{$\Sigma$-algebraically compact modules and $\mathbf L_{\omega _1\omega}$-compact cardinals}

\author{\textsc{Jan \v Saroch}}
\address{Charles University, Faculty of Mathematics and Physics, Department of Algebra \\ 
Sokolovsk\'{a} 83, 186 75 Praha~8, Czech Republic}
\email{saroch@karlin.mff.cuni.cz}

%\author{\textsc{Jan Trlifaj}}
%\address{Charles University, Faculty of Mathematics and Physics, Department of Algebra \\ 
%Sokolovsk\'{a} 83, 186 75 Praha~8, Czech Republic}
%\email{trlifaj@karlin.mff.cuni.cz}
 
%\keywords{($\Sigma$)-algebraically compact module, strongly compact cardinal, $\omega$-measurable cardinal, large ultrapowers, projective modules}

\thanks{The research of the author has been supported by grant GA\v CR 14-15479S}

%\subjclass[2010]{13F07 (primary), 13F10, 13F20, 03H15 (secondary)}
\date{\today}

% = Abstract ========================================================
\begin{abstract} We prove that the property $\Add(M)\subseteq \Prod(M)$ characterizes $\Sigma$-algebraically compact modules if $|M|$ is not $\omega$-measurable. Moreover, under a large cardinal assumption, we show that over any ring $R$ where $|R|$ is not $\omega$-measurable, any free module $M$ of $\omega$-measurable rank satisfies $\Add(M)\subseteq \Prod(M)$, hence the assumption on $|M|$ cannot be dropped in general (e.g. over small non-right perfect rings). In this way, we extend results from a~recent paper \cite{B} by Simion Breaz.\end{abstract} 

\maketitle
\vspace{4ex}

\section{Introduction}
\label{sec:intro}

Algebraically compact modules, and correlatively the $\Sigma$-algebraically compact ones (see Preliminaries section for definitions), represent the cornerstone of model theory of modules. They serve as the sufficiently saturated objects in the area, capturing at the same time nontrivial amount of information on how wildly (or tamely) all modules over a fixed ring behave. Algebraically compact modules (also called pure-injective) have been studied intensively for decades: we refer to \cite{Z} and \cite{P} for an introduction to the topic, \cite{H-Z} for applications to direct sum decompositions and pure-semisimple rings, \cite{AST} for a recent application in the approximation theory of modules, or \cite{G-AH} for a generalization of the concept.

In this context, it seems rather surprising that a completely new characterization of $\Sigma$-algebraically compact modules occured recently in \cite{B}, where the following theorem was proved:

\begin{thm} \label{t:Breaz} %\cite[Theorem 1.1]{B}\ 
Assume that there are no measurable cardinals. Then a module $M$ over a ring $R$ is $\Sigma$-algebraically compact \iff $\Add(M)\subseteq \Prod(M)$.
\end{thm}

In this short note, we give an example showing that the large cardinal assumption cannot be weakened below `no $\mathbf L_{\omega_1\omega}$-compact cardinals' (Section~\ref{sec:indep}). We also generalize Theorem~\ref{t:Breaz} by dropping the large cardinal hypothesis and assuming that $|M|$ is not $\omega$-measurable instead (Section~\ref{sec:refine}).

\smallskip

It remains open whether Theorem~\ref{t:Breaz} holds assuming only the nonexistence of $\mathbf L_{\omega _1\omega}$-compact cardinals.

\bigskip          
\section{Preliminaries}
\label{sec:prelim}

Unless stated otherwise, we work in ZFC. Let $R$ be a (unital, associative) ring. We denote by $\ModR$ the class of all (right $R$-) modules. The elements of this class are the models of the first-order theory of right $R$-modules, i.e. the theory in the language of abelian groups extended by unary function symbols $\cdot r$, for all $r\in R$, which are written in postfix notation. As usual, we typically omit $\cdot$. For a~module $M$ and a set $I$, we denote by $M^{(I)}$ the direct sum of $I$ copies of $M$, i.e. the submodule of the direct power $M^I$ consisting of all elements with finite support.

Probably the most important formulas in model theory of modules are the positive-primitive ones, i.e. the existential formulas whose quantifier-free core is a conjunction of positive atomic formulas. This stems from the result of Baur, Monk and Garavaglia that any complete theory of modules admits elimination of quantifiers up to positive-primitive formulas (pp-formulas, for brevity). The set of all pp-formulas in a fixed number of free variables is naturally ordered by setting $\varphi\leq\psi$ \iff $M\models \varphi\rightarrow\psi$ for all modules $M$.

Each pp-formula $\psi$ with one free variable defines the subfunctor of the forgetful functor from $\ModR$ to Mod-$\mathbb Z$ via the assignment $M\mapsto \psi(M)$. This functor is often called a \emph{finite matrix functor}, and it commutes with direct products and direct sums. 

\smallskip

For a module $N$ and its submodule $M$, we say that $M$ is \emph{pure} in $N$, provided that for any pp-formula $\psi(x_0,\dots,x_{n-1})$, we have $\psi(M)=M^n\cap\psi(N)$. We call $\psi(M) = \psi M = \{\bar m\in M^n\mid M\models \psi(\bar m)\}$ a \emph{pp-definable subgroup} of $M$. The basic examples of pure submodules are direct summands and elementary submodels.

\begin{defn} \label{d:algcomp} We say that an $R$-module $M$ is \emph{algebraically compact} (or \emph{pure-injective}) if one of the following equivalent statements are satisfied:
\begin{enumerate}
\item Whenever $M$ is pure in a module $N$ then it splits, i.e. there exists a~homomorphism $\pi:N\to M$ with $\pi\restriction M = id_M$.

\item Every system of pp-formulas (with one free variable) with parameters from $M$ which is finitely satisfied in $M$ is actually realized in $M$.

\item Every system, in arbitrary many unknowns, of $R$-linear equations with parameters from $M$ which is finitely satisfied in $M$ actually has a solution in $M$.
\end{enumerate}
\end{defn}

It should not come as a surprise that each module can be elementarily embedded in an algebraically-compact one. There even exists a minimal such extension. Algebraically compact modules are closed under taking arbitrary direct products.

Basic examples over the ring $\mathbb Z$ include $\mathbb Q, \mathbb Z_{p^\infty}$ and $\mathbb J_p$ for $p$ prime, where $\mathbb Z_{p^\infty}$ is the Pr\"ufer $p$-group and $\mathbb J_p$ denotes the group of $p$-adic integers.

\begin{defn} \label{d:sigma} Let $M$ be an $R$-module. We say that $M$ is \emph{$\Sigma$-algebraically compact} (or \emph{$\Sigma$-pure-injective}) if one of the following equivalent statements are satisfied:
\begin{enumerate}
\item The module $M^{(\kappa)}$ is algebraically-compact for any/every infinite cardinal~$\kappa$.
%\item The module $M^{(\omega)}$ is algebraically-compact.
\item Every descending chain of pp-definable (by pp-formulas with one free variable) subgroups of $M$ eventually stabilizes.
\item The complete theory Th$(M)$ is totally transcendental.
\end{enumerate}
\end{defn}

From the three examples of algebraically compact abelian groups given above, only $\mathbb J_p$ is not $\Sigma$-algebraically compact. In fact, an abelian group is $\Sigma$-algebraically compact \iff it is the direct sum of a divisible group and a group of bounded exponent. Notice that by $(3)$ from the definition, $\Sigma$-algebraically compact modules are closed under taking elementarily equivalent modules. They are also closed under pure submodules, by $(2)$, and under direct powers. Readers interested in more detailed description of (and relations between) the notions of purity and algebraic compactness can consult Chapter 2 from \cite{GT}.

\smallskip

For a module $M$, we denote by $\Add(M)$ the class of all modules isomorphic to direct summands in $M^{(I)}$, where $I$ is an arbitrary set. For example, if $M$ is a nonzero free module (i.e. isomorphic to $R^{(I)}$ for some nonempty set $I$) then $\Add(M)$ denotes precisely the class of all projective modules.

Dually, we define $\Prod(M)$ as the class of all modules isomorphic to direct summands in $M^I$, where $I$ is arbitrary. If $M$ is $\Sigma$-algebraically compact, then all models of Th$(M)$ are contained in $\Prod(M)$. Since $M^{(I)}$ is always pure in $M^I$, it follows that $\Add(M)\subseteq \Prod(M)$ for every $\Sigma$-algebraically compact module $M$.

\smallskip

We call a ring $R$ \emph{right perfect} if every module pure in a free module splits. For instance, if $R$ is $\Sigma$-algebraically compact as a right $R$-module over itself, then $R$ is right perfect. There are plenty of examples of right perfect rings, e.g. right artinian rings, as well as those which do not satisfy this condition: $\mathbb Z$, $\mathbb Q[[x]]$ etc.

Finally, a ring $R$ is called \emph{right pure-semisimple} if all right modules are $(\Sigma)$-algebraically compact.

\bigskip          
\section{Non-algebraically compact modules satisfying $\Add(M)\subseteq \Prod(M)$}
\label{sec:indep}

In this section, we will deal with several types of large cardinals. All definitions are standard, except maybe for the one of an \emph{$\omega$-measurable cardinal} by which we mean a cardinal greater than or equal to the first measurable cardinal. Equivalently, a~cardinal admitting a nonprincipal $\omega_1$-complete ultrafilter. The following notion was originally defined in terms of infinitary logic.

\begin{defn} \label{d:compact} Let $\kappa, \nu$ be infinite cardinals. Following \cite[Definition 2.17]{EM} (see also \cite{E}), we say that $\kappa$ is \emph{$\mathbf L_{\nu\omega}$-compact} if for every set $I$, every $\kappa$-complete filter on $I$ can be extended to a $\nu$-complete ultrafilter. Moreover, $\kappa$ is called \emph{strongly compact} if it is uncountable and $\mathbf L_{\kappa\omega}$-compact.
\end{defn}

We list some basic properties of $\mathbf L_{\nu\omega}$-compact cardinals in the following

\begin{rem} If there exists an $\mathbf L_{\nu\omega}$-compact cardinal $\kappa$, then $\kappa\geq\nu$ and all cardinals $\mu\geq\kappa$ are $\mathbf L_{\nu\omega}$-compact, too. Moreover, every $\mathbf L_{\nu\omega}$-compact cardinal is $\omega$-measurable and $\mathbf L_{\lambda\omega}$-compact, where $\lambda$ is the first measurable cardinal, provided that $\nu$ is uncountable. The latter follows from the well-known fact that any $\omega_1$-complete ultrafilter has to be $\lambda$-complete, where $\lambda$ is the first measurable cardinal.
\end{rem}

The class of $\mathbf L_{\nu\omega}$-compact cardinals was intensively studied by Eda in \cite{E}. He gave a thorough characterization of these cardinals, besides other things in terms of vanishing of the $\Hom_R(-,R)$-functor in $\ModR$. Recently, $\mathbf L_{\nu\omega}$-compact cardinals, under the name \emph{$\nu$-strongly compact}, have re-emerged in \cite{BM}---mostly in the module-theoretic context again.

\smallskip
The following proposition gives us enough information on the cardinality of large ultrapowers. We prove it along the lines of the classic \cite[Theorems 1.17 and 1.25]{FMS}.

\begin{prop} \label{p:upower} Let $\kappa$ be a regular $\mathbf L_{\nu\omega}$-compact cardinal and $\mu, \lambda$ cardinals \st $\mu=\mu^{<\kappa}, \lambda=\lambda^{<\kappa}$. Then there exists a $\nu$-complete ultrafilter $\mathcal U$ on $\lambda$ \st $|\mu^\lambda/\mathcal U|=\mu^\lambda$.
\end{prop}

\begin{proof} We start by fixing a bijection $b:\lambda\to [\lambda]^{<\kappa}$. For each ordinal $\alpha<\lambda$, we set $L_\alpha = \{\gamma<\lambda \mid \alpha\in b(\gamma)\}$. Using the regularity of $\kappa$, it follows that

$$\mathcal F = \{X\subseteq \lambda \mid \bigl(\exists Y\in [\lambda]^{<\kappa}\bigr)\bigcap_{\alpha\in Y} L_\alpha\subseteq X\}$$
is a $\kappa$-complete filter on $\lambda$. By our assumption on $\kappa$, we can extend the filter $\mathcal F$ to a $\nu$-complete ultrafilter $\mathcal U$.
We show that $|(\mu^{<\kappa})^\lambda/\mathcal U|=\mu^\lambda$.

For each $f\in\mu^\lambda$, we put $h(f) = \langle f\restriction b(\alpha) \mid \alpha < \lambda\rangle _\mathcal U$.
%\medskip
It is easy to see that $h:\mu^\lambda\to (\mu^{<\kappa})^\lambda/\mathcal U$ is one--one. Indeed, for two different $f,g\in \mu^\lambda$, we fix $\alpha<\lambda$ with $f(\alpha)\neq g(\alpha)$, and observe that $L_\alpha\subseteq \{\gamma<\lambda \mid f\restriction~ b(\gamma)\neq g\restriction b(\gamma)\}\in\mathcal U$. So $|(\mu^{<\kappa})^\lambda/\mathcal U|\geq\mu^\lambda$, and since $\mu^{<\kappa}=\mu$, we are done.
\end{proof}

In the proof above, it seemingly appears as if the assumption on $\kappa$ is stronger than actually needed, however by \cite[Theorem 4.7]{BM}, the existence of a $\nu$-complete ultrafilter on $\lambda$ containing all the sets $L_\alpha$ (i.e. $\nu$-complete fine measure on $[\lambda]^{<\kappa}$) is, in fact, equivalent to the statement that each $\kappa$-complete filter on $\lambda$ can be extended to a $\nu$-complete ultrafilter. %$\mathbf L_{\nu\omega}$-compactness.

\begin{thm} \label{t:indep} Let $R$ be a ring, $\nu>|R|$ an uncountable cardinal, and let $\kappa$ be a~regular $\mathbf L_{\nu\omega}$-compact cardinal. Consider the least infinite cardinal $\mu$ \st $\mu = \mu^{<\kappa}$. Then $\Add(R^{(I)})\subseteq \Prod(R^{(I)})$ whenever $|I|\geq \mu$.
\end{thm}

\begin{proof} Put $F=R^{(\mu)}$. Notice that $|F|=\mu$, since $|R|<\nu\leq\kappa\leq\mu$. It is enough to prove that $\Add(F)\subseteq \Prod(F)$; in other words---that $\Prod(F)$ contains all projective modules. Let $\xi$ be an arbitrary cardinal. We have to show that $F^{(\xi)}\in\Prod(F)$.

Choose $\lambda = \lambda ^{<\kappa}$ \st $\mu^\lambda\geq\xi$. By Proposition~\ref{p:upower}, there exists a $\nu$-complete ultrafilter $\mathcal U$ on $\lambda$ with $|F^\lambda/\mathcal U|=\mu^\lambda$. From \cite[Theorem II.3.8]{EM}, we deduce that $F^\lambda/\mathcal U$ is isomorphic to a free direct summand in $F^\lambda$. Moreover, we have $F^\lambda/\mathcal U\cong R^{(\mu^\lambda)}\cong F^{(\mu^\lambda)}$, whence $F^{(\xi)}$ is isomorphic to a direct summand in $F^\lambda$, too.
\end{proof}

\begin{exm} Let $R$ be a ring which is not $\Sigma$-algebraically compact as a right module over itself, and such that $|R|$ is not $\omega$-measurable, e.g. let $R$ be a countable non-right perfect ring, for instance $R = \mathbb Z$. Assume that there exists an $\mathbf L_{\omega_1\omega}$-compact cardinal $\kappa$. Then $\Add(R^{(I)})\subseteq \Prod(R^{(I)})$ whenever $|I|\geq 2^\kappa$, while $R^{(I)}$ is not algebraically compact.

%\smallskip
Magidor showed in \cite{M} that it is consistent (modulo some large cardinal assumption) that the first measurable cardinal is strongly compact. In this situation, we obtain that if $R$ is a ring as above, then $\Add(R^{(I)})\subseteq \Prod(R^{(I)})$ \iff $|I|$ is $\omega$-measurable. For the only-if part, we refer to Section~\ref{sec:refine}.

%\smallskip
On the other hand in \cite{BM}, Bagaria and Magidor construct, relative to the existence of a supercompact cardinal, a model of ZFC in which the first $\mathbf L_{\omega_1\omega}$-compact cardinal is singular (with measurable cofinality).
\end{exm}

\begin{cor} \label{c:proper} Assume that for each uncountable cardinal $\nu$ there exists an $\mathbf L_{\nu\omega}$-compact cardinal, e.g. assume that there is a proper class of strongly compact cardinals. Then for any ring $R$, there is a nonzero free module $F$ \st $\Add(F)\subseteq \Prod(F)$. In particular, over any right non-$\Sigma$-algebraically compact ring, there exists a non-algebraically compact $R$-module $F$ with $\Add(F)\subseteq \Prod(F)$.
\end{cor}

Note that the last inclusion in Corollary~\ref{c:proper} is necessarily strict since otherwise $F$ would be $\Sigma$-algebraically compact (even product-complete, see \cite[Definition 2.34]{GT}).

\smallskip
We finish this section with a short discussion concerning the possibilities to weaken our large cardinal assumption. For the simplicity, let us assume that the ring $R$ does not have $\omega$-measurable cardinality, and that it is \emph{slender}, i.e.\ all homomorphisms $f\in\Hom _R(R^\omega,R)$ have the property that $f(e_n)=0$ for all but finitely many $n<\omega$; here, $e_n$ denotes the element of $R^\omega$ with $e_n(n) = 1$ and $e_n(m)=0$ whenever $m\neq n$.

Using \cite[Theorem II.3.8 and Corollary III.3.6]{EM}, we see that for a nonzero free module $F$, $\Add(F)\subseteq \Prod(F)$ is equivalent to the statement `there exists arbitrarily large ultrapowers of $F$ with respect to $\omega_1$-complete ultrafilters'. If the latter implies the existence of an $\mathbf L_{\omega_1\omega}$-compact cardinal, it would be a strong indication that the large cardinal assumption in the statement of Theorem~\ref{t:Breaz} could be weakened to `no $\mathbf L_{\omega_1\omega}$-compact cardinals'. However, the results in \cite{JS} suggest that one can obtain large ultrapowers with fixed base even without (sufficiently) regular ultrafilters. The downside in the paper cited is that to achieve this, the authors needed a supercompact cardinal for a start. It seems to be an open problem whether one can prove the same without an $\mathbf L_{\omega _1\omega}$-compact cardinal. We formulate it as

\begin{conj} Assume that the existence of measurable cardinals is not refutable in ZFC. Then there is a model of ZFC with no $\mathbf L_{\omega _1\omega}$-compact cardinals where the following holds: there exists $\mu$ \st for all cardinals $\kappa$ there is $\lambda$ and an $\omega _1$-complete ultrafilter $\mathcal U$ on $\lambda$ \st $|\mu^\lambda/\mathcal U|\geq \kappa$.
\end{conj}

\bigskip          
\section{Breaz's theorem for non-$\omega$-measurable modules}
\label{sec:refine}

We shall use the following generalization of \cite[Theorem 2.1]{B} (see also \cite[Theorem~2]{DH-Z} and \cite[Theorem III.3.9]{EM}). In what follows, for a filter $\mathcal D$ on $K$, we denote by $\pi_\mathcal D:\prod _{i\in K}U_i\to\prod _{i\in K}U_i/\mathcal D$ the canonical projection. Recall that $\Ker(\pi_\mathcal D)$ is pure in $\prod _{i\in K}U_i$.

\begin{thm} \label{t:huisgen} Let $R$ be a ring and $\psi_0\geq\psi_1\geq\dotsb$ be a descending chain of pp-formulas with one free variable. Let $f:\prod _{i\in I}U_i\to \bigoplus_{j\in J}V_j$ be a homomorphism of $R$-modules. Then there exist $n_0<\omega$, a finite number of $\omega_1$-complete ultrafilters $\mathcal D_1,\dotsc,\mathcal D_k$ on $I$, and a finite subset $J^\prime$ of $J$ such that

$$f(\psi_{n_0}\bigcap_{n=1}^k \Ker(\pi_{\mathcal D_n}))\subseteq \bigoplus _{j\in J^\prime} V_j + \bigcap_{n<\omega}\psi_n\bigoplus _{j\in J}V_j.$$
\end{thm}

\begin{proof} %By transfinite induction on $|I|$. 
The result is trivial for $I$ finite. If $I$ is countable, we use \cite[Lemma 11]{H-Z}; all ultrafilters appearing in this case are principal.

For the general case, denote by $\mathcal I$ the set of all $T\subseteq I$ for which there exist $n_0<\omega$ and a finite subset $J^\prime$ of $J$ \st
% and ultrafilters  $\mathcal D_1,\dots,\mathcal D_k$ on $T$ such that

$$f(\psi_{n_0}\prod _{i\in T} U_i)\subseteq \bigoplus _{j\in J^\prime} V_j + \bigcap_{n<\omega}\psi_n\bigoplus _{j\in J}V_j. \eqno{(\dagger)}$$
%$$f(\psi_{n_0}\bigcap_{n=1}^k \Ker(\pi_{\mathcal D_n}))\subseteq \bigoplus _{j\in J^\prime} V_j + \bigcap_{n<\omega}\psi_n\bigoplus _{j\in J}V_j.$$

%We already know that $\mathcal I\supseteq [I]^{<\omega}$, and 
It is easy to see that $\mathcal I$ is closed under subsets and finite unions. It is also non-empty and closed under countable unions---we use the same argument as in the proof of \cite[Theorem 2]{DH-Z}:

\smallskip

Let $\bigcup _{m<\omega}Y_m$ be a union of pairwise disjoint subsets of $I$ (not necessarily members of $\mathcal I$). Then we capitalize on the already proven countable case applied to the restriction of $f$: $$\prod _{m<\omega}\Bigl(\prod _{i\in Y_m} U_i\Bigr)\to \bigoplus_{j\in J}V_j.$$
This gives us some $n_0<\omega$ \st $\bigcup _{m\geq n_0} Y_m\in\mathcal I$ (recall that the ultrafilters have been principal). In what follows, let us denote this property of $\mathcal P(I)$ by $(*)$.
\smallskip

%For a contradiction,
We are done if $I\in\mathcal I$, so let us assume that $I\not\in\mathcal I$. Using the property $(*)$, we deduce that each $W\in\mathcal P(I)\setminus\mathcal I$ contains a subset $X\not\in\mathcal I$ \st for all (disjoint) partitions $X_1\cup X_2$ of $X$ precisely one of these sets belongs to $\mathcal I$. Let $\mathcal X$ denote the set of all such sets $X$. Then every $\mathcal Y\subseteq \mathcal X$ whose elements are pairwise disjoint, is finite (again by $(*)$). Let $\mathcal Y = \{Y_1,\dots ,Y_k\}$ be a maximal disjoint subset of $\mathcal X$. It follows that $W=I\setminus\bigcup\mathcal Y\in\mathcal I$; otherwise, an element of $\mathcal X$ would have to be contained in it.

For $n=1,\dots, k$, we put $\mathcal D_n = \{Z\subseteq I\mid Z\cap Y_n\not\in\mathcal I\}$. By the construction, all the $\mathcal D_n$ are $\omega_1$-complete ultrafilters. It is routine to deduce from the property $(*)$ that for all $i = 1,\dots, k$, there exist (uniform) $n_{0,i}<\omega$ and $J_i^\prime\in [J]^{<\omega}$ such that $(\dagger)$ holds with $n_0=n_{0,i},J^\prime = J_i^\prime$ for all $T\in\mathcal P(Y_i)\cap\mathcal I$.

%We claim that, for all $i = 1,\dots, k$, there exist uniform $n_{0,i}<\omega$ and $J_i^\prime\in [J]^{<\omega}$ such that $(\dagger)$ holds with $n_0=n_{0,i},J^\prime = J_i^\prime$ for all $T\in\mathcal P(Y_i)\cap\mathcal I$. This, however, routinely follows from the property $(*)$.

If we denote by $n_{0,0}$ and $J_0^\prime$ the $n_0$ and $J^\prime$, respectively, given by $(\dagger)$ for $W\in\mathcal I$, our desired $n_0$ and $J^\prime$ are defined as $\max\{n_{0,i}\mid i = 0,\dots, k\}$ and $\bigcup _{i=0}^k J_i^\prime$.

%It follows that $\mathcal I_X = \mathcal I\cap \mathcal P(X)$ is an $\omega _1$-complete prime ideal on $X$. The corresponding dual ultrafilter witnesses that $X\in\mathcal I$, a contradiction.
\end{proof}

The following proposition is an enhancement of \cite[Proposition 2.3]{B}.

\begin{prop} \label{p:enhance} Let $\langle U_i \mid i\in I\rangle$ and $\langle V_j \mid j\in J\rangle$ be two sequences of modules, $J$ infinite. Assume that there is a non-$\omega$-measurable cardinal $\kappa$ such that $|J|>\sup _{i\in I}|U_i| = \kappa$. Let $f:\prod _{i\in I}U_i\to \bigoplus_{j\in J}V_j$ be an epimorphism.

If $\langle \psi_n \mid n<\omega\rangle$ is a descending sequence of pp-formulas with one free variable \st $f\restriction \psi_n(\prod _{i\in I}U_i)$ is onto $\psi_n(\bigoplus_{j\in J}V_j)$ for infinitely many $n<\omega$, then there exists an infinite $L\subseteq J$ \st for every $j\in L$ the sequence $\psi_0(V_j)\supseteq \psi_1(V_j)\supseteq\psi_2(V_j)\supseteq\dotsb$ eventually stabilizes.
\end{prop}

\begin{proof} Using Theorem~\ref{t:huisgen}, there exist $J^\prime\in [J]^{<\omega}$, $n_0<\omega$ and $\omega _1$-complete ultrafilters $\mathcal D_1,\dotsc,\mathcal D_k$ on $I$ \st

$$f(\psi_{n_0}\bigcap_{n=1}^k \Ker(\pi_{\mathcal D_n}))\subseteq \bigoplus _{j\in J^\prime} V_j + \bigcap_{n<\omega}\psi_n\bigoplus _{j\in J}V_j.$$

\medskip

By our hypothesis, we can w.l.o.g.\ assume that the map $f\restriction \psi_{n_0}(\prod _{i\in I}U_i)$ is onto $\psi_{n_0}(\bigoplus_{j\in J}V_j)$. Since $f(\psi _{n_0}(\prod _{i\in I}U_i))\subseteq \psi _{n_0}(\bigoplus_{j\in J}V_j)$, we obtain

$$f(\psi_{n_0}\bigcap_{n=1}^k \Ker(\pi_{\mathcal D_n}))\subseteq \Biggl(\bigoplus _{j\in J^\prime} V_j + \bigcap_{n<\omega}\psi_n\bigoplus _{j\in J}V_j\Biggr)\cap \psi _{n_0}\bigoplus_{j\in J}V_j,$$
and from this point on, we use the same computation as in the proof of \cite[Proposition 2.3]{B} to deduce

$$f(\psi_{n_0}\bigcap_{n=1}^k \Ker(\pi_{\mathcal D_n}))\subseteq \psi _{n_0}\left(\bigoplus _{j\in J^\prime} V_j\right) + \bigoplus _{j\in J\setminus J^\prime}\left(\bigcap _{n<\omega} \psi _n(V_j)\right),$$
and further that

$$\left| \bigoplus _{j\in J\setminus J^\prime}\frac{\psi _{n_0} (V_j)}{\bigcap_{n<\omega}\psi _n(V_j)} \right| \leq\left|\frac{\psi_{n_0}\bigl(\prod _{i\in I}U_i\bigr)}{\psi _{n_0}\bigcap_{n=1}^k \Ker(\pi_{\mathcal D_n})}\right|.$$

Now, we observe that $\bigcap_{n=1}^k \Ker(\pi_{\mathcal D_n}) = \Ker(\pi_{\mathcal D})$ where $$\mathcal D = \Bigl\{\bigcup _{n=1}^k e(n)\mid e\in\prod _{n=1}^k \mathcal D_n\Bigr\}$$ is the filter on $I$.
It follows that $\Ker(\pi_{\mathcal D})$ is a pure submodule in $\prod _{i\in I}U_i$, and so $\psi _{n_0}\Ker(\pi_{\mathcal D}) = \Ker(\pi_{\mathcal D})\cap\psi _{n_0}\bigl(\prod _{i\in I}U_i\bigr)$. Thus we have, using Second Isomorphism Theorem,

$$\left| \bigoplus _{j\in J\setminus J^\prime}\frac{\psi _{n_0} (V_j)}{\bigcap_{n<\omega}\psi _n(V_j)} \right| \leq \left|\frac{\Ker(\pi_{\mathcal D}) + \psi _{n_0}\bigl(\prod _{i\in I}U_i\bigr)}{\Ker(\pi_{\mathcal D})}\right| \leq \left|\prod _{i\in I}U_i/\mathcal D\right|.$$

Since all the ultrafilters $\mathcal D_n$ are $\kappa^+$-complete ($\kappa$ is not $\omega$-measurable), we have $|\prod _{i\in I}U_i/\mathcal D_n|\leq |\kappa^I/\mathcal D_n|=\kappa$ for all $n=1,2,\dotsc, k$. It immediately yields that $\left|\prod _{i\in I}U_i/\mathcal D\right| \leq \kappa^k$; note that $\kappa$ can be finite. We conclude that

$$\left| \bigoplus _{j\in J\setminus J^\prime}\frac{\psi _{n_0} (V_j)}{\bigcap_{n<\omega}\psi _n(V_j)} \right|\leq \kappa^k < |J\setminus J^\prime|,$$
which readily implies the existence of a desired infinite $L \subseteq J$.
\end{proof}

We are in a position to state the main result of this section.

\begin{thm} \label{t:main2} Let $R$ be an arbitrary ring and $M$ an $R$-module \st $|M|$ is not $\omega$-measurable. Then $\Add(M)\subseteq \Prod(M)$ \iff $M$ is $\Sigma$-algebraically compact.
\end{thm}

\begin{proof} The if part follows easily from the purity of the embedding $M^{(\lambda)}\subseteq M^\lambda$ for all cardinals $\lambda$.

For the only-if part, let $J$ be an arbitrary infinite set with $|J|>|M|$. By our assumption, there is a split epimorphism $f:M^I\to M^{(J)}$ for some $I$. To prove that $M$ is $\Sigma$-algebraically compact, it is necessary and sufficient that for each descending chain $\langle \psi_n \mid n<\omega\rangle$ of pp-formulas with one free variable, the chain $\psi_0(M)\supseteq \psi_1(M)\supseteq\psi_2(M)\supseteq\dotsb$ of pp-definable subgroups eventually stabilizes. This follows from Proposition~\ref{p:enhance} applied for $U_i=V_j=M$ for all $i\in I, j\in J$, since $f\restriction \psi_n(M^I)$ is onto $\psi_n(M^{(J)})$, for all $n<\omega$, using the fact that $f$ splits.
\end{proof}

We finish by an instance of \cite[Corollary 1.2]{B}:

\begin{cor} \label{c:end} Let $M$ be an $R$-module such that $\Prod(M)=\ModR$ and $|M|$ is not $\omega$-measurable. Then $R$ is right pure-semisimple.
\end{cor}

\begin{proof} By our assumption, the inclusion $\Add(M)\subseteq\Prod(M)$ trivially holds, and so we can use Theorem~\ref{t:main2} to deduce that $M$ is $\Sigma$-algebraically compact. It follows that all modules are $\Sigma$-algebraically compact.
\end{proof}

\bigskip

\emph{Acknowledgement}: I would like to thank Simion Breaz who drew my attention to this interesting topic.

% = Bibliography ==============================================================


\bigskip
\begin{thebibliography}{AHT}

\bibitem{AST} L.\ Angeleri H\" ugel, J.\ \v Saroch, J.\ Trlifaj,
	\emph{Approximations and Mittag-Leffler conditions},
	preprint.

%\bibitem{BS} S.\ Bazzoni, J.\ \v S\v tov\'\i\v cek,
%  \emph{Flat Mittag-Leffler modules over countable rings},
%  Proc.\ Amer.\ Math.\ Soc. \textbf{140}\ (2012), 1527--1533.

%\bibitem{Be} J.\ L.\ Bell,
%  \emph{On compact cardinals},
%  Zeitschr.\ f.\ math.\ Logik und Grundlagen d.\ Math.\ \textbf{20}\ (1974), 389--393.

\bibitem{BM} J.\ Bagaria, M.\ Magidor,
  \emph{Group radicals and strongly compact cardinals},
  Trans.\ Amer.\ Math.\ Soc.\ \textbf{366} (2014), no.\ 1, 1857--1877.

\bibitem{B} S.\ Breaz,
  \emph{$\Sigma$-pure-injectivity and Brown representability},
  to appear in Proc.\ Amer.\ Math.\ Soc.

\bibitem{DH-Z} M.\ Dugas, M.\ B.\ Zimmermann-Huisgen,
  \emph{Iterated direct sums and products of modules},
  in Abelian group theory, Proc. Oberwolfach Conf. 1981, Lect. Notes Math. \textbf{874} (1981), 179--193.

\bibitem{E} K.\ Eda,
  \emph{Slender modules, endo-slender abelian groups and large cardinals},
  Fund.\ Math.\ \textbf{135} (1990), 5--24.

\bibitem{EM} P.\ C.\ Eklof, A.\ H.\ Mekler, 
  \emph{Almost Free Modules},
  Revised ed., North--Holland, New York 2002.

\bibitem{FMS} T.\ Frayne, A.\ C.\ Morel, D.\ S.\ Scott,
  \emph{Reduced direct products},
  Fund.\ Math.\ \textbf{51} (1962), 195--228.

\bibitem{GT} R.\ G\"obel, J.\ Trlifaj,
	\emph{Approximations and Endomorphism Algebras of Modules},
	de Gruyter Expositions in Mathematics \textbf{41},
	2nd revised and extended edition, Berlin-Boston 2012.

\bibitem{G-AH} P.\ A.\ Guil Asensio, Ivo Herzog,
  \emph{Model-theoretic aspects of $\Sigma$-cotorsion modules},
  Ann. Pure Appl. Logic \textbf{146} (2007), no. 1, 1--12.

\bibitem{H-Z} B.\ Huisgen-Zimmermann,
  \emph{Purity, algebraic compactness, direct sum decompositions, and representation type},
  Krause, Henning (ed.) et al., Infinite length modules. %Proceedings of the conference, Bielefeld, Germany, September 7-11, 1998.
  Trends Math., Basel, 2000, 331--367.

\bibitem{JS} R.\ Jin, S.\ Shelah,
  \emph{Possible size of an ultrapower of $\omega$},
  Arch.\ Math.\ Logic \textbf{38} (1999), 61--77.

\bibitem{M} M.\ Magidor,
  \emph{How large is the first strongly compact cardinal? or A study on identity crises},
   Ann.\ Math.\ Logic \textbf{10} (1976), no. 1, 33--57.

\bibitem{P} M.\ Prest,
	\emph{Model Theory and Modules},
	London Math. Soc. Lec. Note Ser. \textbf{130},
	Cambridge University Press, Cambridge, 1988.

\bibitem{Z} M.\ Ziegler,
	\emph{Model theory of modules},
	Ann.\ Pure Appl.\ Logic \textbf{26} (1984), 149--213.
\end{thebibliography}
\end{document}